\documentclass[reqno]{amsart}
\usepackage{amsfonts}
\usepackage[reqno]{amsmath}
\usepackage{amscd}
\usepackage{amssymb}
\usepackage{graphicx}

\providecommand{\U}[1]{\protect \rule{.1in}{.1in}}
\newtheorem{theorem}{Theorem}
\newtheorem{corollary}[theorem]{Corollary}

\theoremstyle{definition}
\newtheorem{definition}[theorem]{Definition}

\theoremstyle{remark}
\newtheorem{remark}{Remark}

\input{tcilatex}

\begin{document}
\title[Biorthogonal polynomials in two variables]{Finite Bivariate
Biorthogonal M - Konhauser Polynomials}
\author{Esra G\"{U}LDO\u{G}AN LEKES\.{I}Z}
\address{Department of Software Engineering, Faculty of Engineering, Ostim
Technical University, Ankara 06374, T\"{u}rkiye}
\email{esra.guldoganlekesiz@ostimteknik.edu.tr}
\author{Bayram \c{C}EK\.{I}M}
\address{Department of Mathematics, Faculty of Science, Gazi University,
Ankara 06500, T\"{u}rkiye}
\email{bayramcekim@gazi.edu.tr}
\author{Mehmet Ali \"{O}ZARSLAN}
\address{Department of Mathematics,Faculty of Arts and Sciences, Eastern
Mediterranean University, Gazimagusa, TRNC, Mersin 10, T\"{u}rkiye}
\email{mehmetali.ozarslan@emu.edu.tr}
\subjclass{}
\keywords{Biorthogonal polynomial, finite orthogonal polynomial, Konhauser
polynomial, generating function, Laplace transform, fractional derivative,
fractional integral}

\begin{abstract}
In this paper, we construct the pair of finite bivariate biorthogonal M -
Konhauser polynomials, reduced to the finite orthogonal polynomials $%
M_{k}^{\left( p,q\right) }\left( y\right) $, by choosing appropriate
parameters in order to obtain a relation between the Jacobi Konhauser
polynomials and this new finite bivariate biorthogonal polynomials$\
_{K}M_{k;\upsilon }^{\left( p,q\right) }\left( y,z\right) $ similar to the
relation between the classical Jacobi polynomials $P_{k}^{\left( p,q\right)
}\left( y\right) $ and the finite orthogonal polynomials $M_{k}^{\left(
p,q\right) }\left( y\right) $. Several properties like generating function,
operational/integral representation are derived and some applications like
fractional calculus, Fourier transform and Laplace transform are studied
thanks to that new transition relation and the definition of finite
bivariate M - Konhauser polynomials.
\end{abstract}

\maketitle

\section{Introduction}

The theory of orthogonal polynomials is a field that has been of interest to
many researchers over the years \cite%
{LN,Gau,Gau2,Koelink,KM2,Jamei2,AAT,GAA2,GAA}. In addition to having wide
areas in mathematics, it has also many uses in branches of physics such as
quantum mechanics and electrostatics.

In the last two decades, finite orthogonal polynomials have become a new
research area. In addition to previous studies, there are also many new
studies \cite{SMA,MSAN,GL2,GL3,GL4,GL5,Jamei,KM,MMH,GL}. It should be noted
that if there are some parametric constraints in orthogonal polynomials,
then these classes are said to have a finite orthogonality.

Interest in expanding orthogonal polynomials has been increasing in recent
years. Therefore, orthogonal polynomials have been expanded to concepts such
as biorthogonal forms, multivariate analogs, orthogonal matrix extension or
d-orthogonal polynomials. For example, as an example of the entension of
orthogonal polynomials, Koornwinder proved a very useful and general method
in order to obtain bivariate orthogonal polynomials\textbf{\ }and introduced
some examples\textbf{\ }\cite{Koornwinder}. In the following years,
mathematicians obtained new families and studied on polynomials of
Koornwinder type \cite{GL2,GL3,GL4,GL5,GL,FPP,FPP2,MMPP,MMPP2,MPP}. In fact,
in \cite{GL}, G\"{u}ldo\u{g}an Lekesiz gave a generalization of bivariate
orthogonal polynomials over the parabolic biangle. On the other hand,
biorthogonal polynomials are introduced \cite%
{Konhauser2,Konhauser,MT,MT2,MT0,Carlitz,AlSalamVerma}. Then they are
expanded to bivariate and multivariate analogs \cite%
{BinSaad,OzKurt,OzEl,EBM,OzEl2}.

In this study, it is aimed to define a finite set of biorthogonal
polynomials, which has attracted attention in recent years, by extending
them to two variables. This novel family will also have many potential
fields of application.

First, let us give some basic definitions and theorems.

\begin{definition}
If%
\begin{equation*}
J_{r,k}=\int\limits_{d_{1}}^{d_{2}}\rho \left( y\right)
G_{r}(y)Q_{k}(y)dy=\QATOPD\{ . {\ \ \ 0;\ \ r\neq k}{\neq 0;\ \ r=k,},\
r,k\in 
%TCIMACRO{\U{2115} }%
%BeginExpansion
\mathbb{N}
%EndExpansion
_{0},
\end{equation*}%
then $\{G_{r}(y)\}$ and $\{Q_{k}(y)\}$ are biorthogonal polynomials
corresponding to the weight function $\rho (y)$ and the basic polynomials $%
g(y)$ and $q(y)$, on $(d_{1},d_{2})$ \cite{Konhauser2}.
\end{definition}

Equivalently, the definition of the biorthogonality has been given by the
following theorem.

\begin{theorem}
Assume that $\rho (y)$ is a weight function over the interval $(d_{1},d_{2})$%
, and $g(y)$ and $q(y)$ are fundamental polynomials. For $r,k\in 
%TCIMACRO{\U{2115} }%
%BeginExpansion
\mathbb{N}
%EndExpansion
_{0}$, the necessary and sufficient condition is%
\begin{equation*}
J_{r,k}=\int\limits_{d_{1}}^{d_{2}}\rho \left( y\right)
G_{r}(y)Q_{k}(y)dy=\QATOPD\{ . {\ \ \ 0;\ \ r\neq k}{\neq 0;\ \ r=k}
\end{equation*}%
so that%
\begin{equation*}
\int\limits_{d_{1}}^{d_{2}}\rho \left( y\right) \left[ g\left( y\right) %
\right] ^{j}Q_{k}(y)dy=\QATOPD\{ . {\ \ 0;\ j=0,1,...,k-1}{\neq 0;\ j=k\ \ \
\ \ \ \ \ \ \ \ \ \ \ \ }
\end{equation*}%
and%
\begin{equation*}
\int\limits_{d_{1}}^{d_{2}}\rho \left( y\right) \left[ q\left( y\right) %
\right] ^{j}G_{r}(y)dy=\QATOPD\{ . {\ \ 0;\ j=0,1,...,r-1}{\neq 0;\ j=r\ \ \
\ \ \ \ \ \ \ \ \ \ \ \ },
\end{equation*}%
are provided \cite{Konhauser2}.
\end{theorem}

One of the well-known examples of univariate analogs \cite{Konhauser2,
Konhauser, MT, MT2, MT0, Carlitz, AlSalamVerma} of biorthogonal polynomials
is the Konhauser polynomials suggested by the generalized Laguerre
polynomials defined by Konhauser \cite{Konhauser} as follow:%
\begin{equation}
Z_{k}^{\left( c\right) }\left( y;\upsilon \right) =\frac{\Gamma \left(
\upsilon k+c+1\right) }{k!}\sum\limits_{j=0}^{k}\left( -1\right) ^{j}\binom{k%
}{j}\frac{y^{\upsilon j}}{\Gamma \left( \upsilon j+c+1\right) }  \label{Zdef}
\end{equation}%
and%
\begin{equation}
Y_{k}^{\left( c\right) }\left( y;\upsilon \right) =\frac{1}{k!}%
\sum\limits_{i=0}^{k}\frac{y^{i}}{i!}\sum\limits_{j=0}^{i}\left( -1\right)
^{j}\binom{i}{j}\left( \frac{1+c+j}{\upsilon }\right) _{k},  \label{Ydef}
\end{equation}%
where $\Gamma \left( .\right) $ and $\left( .\right) _{k}$ are well-known
Gamma function and Pochhammer symbol, respectively, $c>-1$ and $\upsilon $
is a positive integer.

The pair of biorthogonal polynomials suggested by the Laguerre polynomials
are biorthogonal with respect to $\rho \left( y\right) =y^{c}e^{-y}$ over $%
\left( 0,\infty \right) $, and satisfy the biorthogonality relation%
\begin{equation}
\int\limits_{0}^{\infty }e^{-y}y^{c}Z_{k}^{\left( c\right) }\left(
y;\upsilon \right) Y_{r}^{\left( c\right) }\left( y;\upsilon \right) dy=%
\frac{\Gamma \left( \upsilon k+c+1\right) }{k!}\delta _{k,r},  \label{3}
\end{equation}%
where $\delta _{k,r}$ is Kronecker delta.

Other important examples are the biorthogonal polynomials which are
suggested by the Jacobi polynomials \cite{MT}, the Hermite polynomials \cite%
{MT2} and the Szeg\"{o}-Hermite polynomials \cite{MT0} by Madhekar and
Thakare.\bigskip

As more general, an example of bivariate biorthogonal polynomials, namely
the Laguerre-Konhauser polynomials, is first introduced by Bin-Saad \cite%
{BinSaad}:%
\begin{align*}
\ _{\upsilon }L_{k}^{\left( c_{1},c_{2}\right) }\left( y,z\right) &
=k!\sum\limits_{l=0}^{k}\sum\limits_{s=0}^{k-l}\frac{\left( -1\right)
^{l+s}y^{s+c_{1}}z^{\upsilon l+c_{2}}}{l!s!\left( k-l-s\right) !\Gamma
\left( s+c_{1}+1\right) \Gamma \left( \upsilon l+c_{2}+1\right) } \\
& =k!\sum\limits_{l=0}^{k}\frac{\left( -1\right)
^{l}y^{l+c_{1}}z^{c_{2}}Z_{k-l}^{\left( c_{2}\right) }\left( z;\upsilon
\right) }{l!\Gamma \left( l+c_{1}+1\right) \Gamma \left( \upsilon k-\upsilon
l+c_{2}+1\right) } \\
& =k!\sum\limits_{l=0}^{k}\frac{\left( -1\right) ^{l}y^{c_{1}}z^{\upsilon
l+c_{2}}L_{k-l}^{\left( c_{1}\right) }\left( y\right) }{l!\Gamma \left(
k-l+c_{1}+1\right) \Gamma \left( \upsilon l+c_{2}+1\right) },
\end{align*}%
where $L_{k}^{\left( c_{1}\right) }\left( y\right) $ are generalized
Laguerre polynomials and $Z_{k}^{\left( c_{2}\right) }\left( z;\upsilon
\right) $ are Konhauser polynomials, for $c_{1},c_{2}>-1$ and $\upsilon
=1,2,...$ .

Later, the second set$\ _{\upsilon }\tciLaplace _{k}^{\left(
c_{1},c_{2}\right) }\left( y,z\right) $ of the form%
\begin{equation*}
\ _{\upsilon }\tciLaplace _{k}^{\left( c_{1},c_{2}\right) }\left( y,z\right)
=L_{k}^{\left( c_{1}\right) }\left( y\right)
\sum\limits_{l=0}^{k}Y_{l}^{\left( c_{2}\right) }\left( z;\upsilon \right) 
\end{equation*}%
is introduced by \"{O}zarslan and K\"{u}rt and that the sets$\ _{\upsilon
}\tciLaplace _{k}^{\left( c_{1},c_{2}\right) }\left( y,z\right) $ and$\
_{\upsilon }L_{k}^{\left( c_{1},c_{2}\right) }\left( y,z\right) $ are
biorthonormal with respect to $\rho \left( y,z\right) =e^{-\left( y+z\right)
}$ on $\left( 0,\infty \right) \times \left( 0,\infty \right) $ is shown in 
\cite{OzKurt}. Further, a new set of bivariate Mittag-Leffler functions $%
E_{c_{1},c_{2},\upsilon }^{(\alpha )}\left( y,z\right) $ corresponding to
the generalized Laguerre-Konhauser polynomials $_{\upsilon }L_{k}^{\left(
c_{1},c_{2}\right) }\left( y,z\right) $ is defined as%
\begin{equation*}
E_{c_{1},c_{2},\upsilon }^{(\alpha )}\left( y,z\right)
=\sum\limits_{m=0}^{\infty }\sum\limits_{l=0}^{\infty }\frac{\left( \alpha
\right) _{m+l}y^{l}z^{\upsilon m}}{m!l!\Gamma \left( c_{1}+l\right) \Gamma
\left( c_{2}+\upsilon m\right) },
\end{equation*}%
where $\alpha ,c_{1},c_{2},\upsilon \in 
%TCIMACRO{\U{2102} }%
%BeginExpansion
\mathbb{C}
%EndExpansion
,\ \func{Re}\left( \alpha \right) ,\func{Re}\left( c_{1}\right) ,\func{Re}%
\left( c_{2}\right) ,\func{Re}\left( \upsilon \right) >0$, and so%
\begin{equation*}
\ _{\upsilon }L_{k}^{\left( c_{1},c_{2}\right) }\left( y,z\right)
=y^{c_{1}}z^{c_{2}}E_{c_{1}+1,c_{2}+1,\upsilon }^{(-k)}\left( y,z\right) 
\end{equation*}%
is presented between the generalized Laguerre-Konhauser polynomials and this
new set of bivariate Mittag-Leffler functions \cite{OzKurt}.

\bigskip

Recently, \"{O}zarslan and Elidemir \cite{OzEl} has given a method for
constructing biorthogonal polynomials in two variables in the following
theorem:

\begin{theorem}
Assume that $G_{r}\left( y\right) $ and $Q_{k}\left( y\right) $ are
biorthogonal polynomials of the fundamental polynomials $g\left( y\right) $
and $q\left( y\right) $, respectively, with respect to the weight function $%
\rho \left( y\right) $ over $\left( \alpha _{1},\alpha _{2}\right) $. Thus,
they satisfy the biorthogonality relation%
\begin{equation*}
\int\limits_{\alpha _{1}}^{\alpha _{2}}\rho \left( y\right) G_{r}\left(
y\right) Q_{k}\left( y\right) dy=J_{k,r}:=\left\{ \QATOP{0,\ k\neq r}{%
J_{k,k},\ \ \ k=r}\right. .
\end{equation*}%
Also, let%
\begin{equation*}
D_{k}\left( y\right) =\sum\limits_{i=0}^{k}E_{k,i}\left( d\left( y\right)
\right) ^{i}
\end{equation*}%
with%
\begin{equation*}
\int\limits_{\beta _{1}}^{\beta _{2}}\rho \left( y\right) D_{r}\left(
y\right) D_{k}\left( y\right) dy=\left\Vert D_{k}\right\Vert ^{2}\delta
_{k,r}.
\end{equation*}%
Then the bivariate polynomials%
\begin{equation}
P_{k}\left( y,z\right) =\sum\limits_{s=0}^{k}\frac{E_{k,s}}{J_{k-s,k-s}}%
\left( g\left( y\right) \right) ^{s}G_{k-s}\left( z\right)   \label{PolP}
\end{equation}%
and%
\begin{equation}
F_{k}\left( y,z\right) =D_{k}\left( y\right)
\sum\limits_{j=0}^{k}Q_{j}\left( z\right)   \label{PolQ}
\end{equation}%
are biorthogonal with respect to the weight function $\rho \left( y\right)
\rho \left( z\right) $ over $\left( \alpha _{1},\alpha _{2}\right) \times
\left( \beta _{1},\beta _{2}\right) $, \cite{OzEl}.\bigskip 
\end{theorem}

By using above method, three new family of bivariate biorthogonal
polynomials are introduced. For $c_{1}>-1$ and $\upsilon =1,2,...$, the
first is the bivariate biorthogonal Hermite Konhauser polynomials \cite{OzEl}
defined as%
\begin{equation}
\ _{\upsilon }H_{k}^{\left( c_{1}\right) }\left( y,z\right)
=\sum\limits_{l=0}^{\left[ k/2\right] }\sum\limits_{s=0}^{k-l}\frac{\left(
-1\right) ^{l}\left( -k\right) _{2l}\left( -k\right) _{l+s}\left( 2y\right)
^{k-2l}z^{\upsilon s}}{\left( -k\right) _{l}\Gamma \left( \upsilon
s+c_{1}+1\right) l!s!}  \label{polH}
\end{equation}%
and%
\begin{equation*}
F_{k}\left( y,z\right) =H_{k}\left( y\right)
\sum\limits_{j=0}^{k}Y_{j}^{\left( c_{1}\right) }\left( z;\upsilon \right) ,
\end{equation*}%
where $H_{k}\left( y\right) $ are the classical orthogonal Hermite
polynomials and $Y_{j}^{\left( c_{1}\right) }\left( z;\upsilon \right) $ are
called as the Konhauser polynomials.

The second one is the set of finite bivariate biorthogonal I - Konhauser
polynomials defined by G\"{u}ldo\u{g}an Lekesiz et. al \cite{EBM}\ as follow.

For $\upsilon =1,2,...$, pair of the finite bivariate I - Konhauser
polynomials $_{K}I_{k;\upsilon }^{\left( p,q\right) }\left( y,z\right) $
defined by%
\begin{equation*}
\ _{K}I_{k;\upsilon }^{\left( p,q\right) }\left( y,z\right) =\frac{\left(
1-p\right) _{k}}{\left( -1\right) ^{k}}\sum_{l=0}^{\left[ k/2\right]
}\sum_{s=0}^{k-l}\frac{\left( -1\right) ^{l}\left( -k\right) _{2l}\left(
-\left( k-l\right) \right) _{s}\left( 2y\right) ^{k-2l}z^{\upsilon s}}{%
\left( p-k\right) _{l}\Gamma \left( \upsilon s+q+1\right) l!s!}
\end{equation*}%
and the polynomials$\ _{K}\mathcal{I}_{k;\upsilon }^{\left( p,q\right)
}\left( y,z\right) $ defined by%
\begin{equation*}
F_{k}\left( y,z\right) :=\ _{K}\mathcal{I}_{k;\upsilon }^{\left( p,q\right)
}\left( y,z\right) =I_{k}^{\left( p\right) }\left( y\right)
\sum\limits_{j=0}^{k}Y_{j}^{\left( q\right) }\left( z;\upsilon \right) ,
\end{equation*}%
together are biorthogonal polynomials with respect to $\rho \left(
y,z\right) =e^{-z}z^{q}\left( 1+y^{2}\right) ^{-\left( p-1/2\right) }$ on $%
\left( -\infty ,\infty \right) \times \left( 0,\infty \right) $ and the
finite biorthogonality relation is%
\begin{align}
& \int\limits_{-\infty }^{\infty }\int\limits_{0}^{\infty }e^{-z}z^{q}\left(
1+y^{2}\right) ^{-\left( p-1/2\right) }\ _{K}I_{k;\upsilon }^{\left(
p,q\right) }\left( y,z\right) \ _{K}\mathcal{I}_{r;\upsilon }^{\left(
p,q\right) }\left( y,z\right) dzdy  \label{IKort} \\
& =\frac{k!2^{2k-1}\sqrt{\pi }\Gamma ^{2}\left( p\right) \Gamma \left(
2p-2k\right) }{\left( p-k-1\right) \Gamma \left( p-k\right) \Gamma \left(
p-k+1/2\right) \Gamma \left( 2p-k-1\right) }\delta _{k,r},  \notag
\end{align}%
where $r,k=0,1,...,R<p-1,\ q>-1$.

On the other hand, \"{O}zarslan and Elidemir \cite{OzEl2}\ derive the third
one called as the bivariate biorthogonal Jacobi Konhauser polynomials as
follow. For $p>-1,q>-1$ and $\upsilon =1,2,...$, pair of the bivariate
Jacobi Konhauser polynomials$\ _{K}P_{k}^{\left( p,q\right) }\left(
y,z\right) $ defined by%
\begin{equation*}
_{K}P_{k}^{\left( p,q\right) }\left( y,z\right) =\frac{\Gamma \left(
k+p+1\right) }{k!}\sum_{l=0}^{k}\sum_{s=0}^{k-l}\frac{\left( -k\right)
_{l+s}\left( k+p+q+1\right) _{l}\left( \frac{1-y}{2}\right)
^{k-2l}z^{\upsilon s}}{\Gamma \left( l+p+1\right) \Gamma \left( \upsilon
s+q+1\right) l!s!}
\end{equation*}%
and the polynomials$\ F_{k}\left( y,z\right) $ defined by%
\begin{equation*}
F_{k}\left( y,z\right) =P_{k}^{\left( p,q\right) }\left( y\right)
\sum\limits_{j=0}^{k}Y_{j}^{\left( q\right) }\left( z;\upsilon \right) ,
\end{equation*}%
together are biorthogonal polynomials with respect to $\rho \left(
y,z\right) =e^{-z}z^{q}\left( 1-y\right) ^{p}\left( 1+y\right) ^{q}$ on $%
\left( -1,1\right) \times \left( 0,\infty \right) $ and the biorthogonality
relation is%
\begin{align*}
& \int\limits_{-1}^{1}\int\limits_{0}^{\infty }e^{-z}z^{q}\left( 1-y\right)
^{p}\left( 1+y\right) ^{q}\ _{K}P_{k}^{\left( p,q\right) }\left( y,z\right)
\ F_{r}\left( y,z\right) dzdy \\
& =\frac{2^{k+p+q+1}\Gamma \left( k+p+1\right) \Gamma \left( k+q+1\right) }{%
\left( 2k+p+q+1\right) \Gamma \left( k+p+q+1\right) k!}\delta _{k,r}.
\end{align*}

With the motivation of papers \cite{OzKurt}, \cite{OzEl}, \cite{EBM}\ and 
\cite{OzEl2},\ we derive a new family of finite biorthogonal polynomials in
two variables. The aim of this paper is to define this new finite set in a
way that enables a transition with the bivariate biorthogonal Jacobi
Konhauser polynomials. Thus, we obtain some properties of M - Konhauser
polynomials via properties of Jacobi Konhauser polynomials like generating
function, integral and operational representation. Further, Laplace
transform, fractional calculus operators and Fourier transform of the finite
biorthogonal M - Konhauser polynomials are investigated.

The article is organized as follows: in the first section, we remind the
general method applied to a univariate biorthogonal and a univariate
orthogonal polynomials in order to construct bivariate biorthogonal
polynomials. In section 2, by choosing appropriate parameters to enable a
connection between bivariate Jacobi Konhauser and finite M - Konhauser
polynomials, we construct a pair of finite biorthogonal polynomials in two
variables,$\ _{K}M_{k;\upsilon }^{\left( p,q\right) }\left( y,z\right) $ and$%
\ _{K}\mathcal{M}_{k;\upsilon }^{\left( p,q\right) }\left( y,z\right) $.
Then, we give a relation between the finite bivariate M - Konhauser
polynomials$\ _{K}M_{k;\upsilon }^{\left( p,q\right) }\left( y,z\right) $
and the bivariate Jacobi Konhauser polynomials$\ _{K}P_{q;\upsilon }^{\left(
\gamma _{1};\gamma _{2};\gamma _{3}\right) }\left( y,z\right) $. With the
help of this relation, we investigate their generating function, integral
and operational representation, Laplace transform and fractional calculus
operators. Also, a partial differential equation is obtained for polynomials$%
\ _{K}\mathcal{M}_{k;\upsilon }^{\left( p,q\right) }\left( y,z\right) $.
Furthermore, to derive a new family of finite biorthogonal functions, we
compute Fourier transform of the finite bivariate biorthogonal M - Konhauser
polynomials.

\bigskip

Before giving definition of the finite bivariate biorthogonal M - Konhauser
polynomials introduced using the set of the finite orthogonal polynomials $%
M_{k}^{\left( p,q\right) }\left( y\right) $ and the Konhauser polynomials,
we remind the finite univariate orthogonal polynomials $M_{k}^{\left(
p,q\right) }\left( y\right) $.\bigskip 

Polynomials $M_{k}^{\left( p,q\right) }\left( y\right) $, which is one of
the three finite solutions of the equation%
\begin{equation*}
y\left( y+1\right) M_{k}^{\prime \prime }\left( y\right) +\left( \left(
2-p\right) y+1+q\right) M_{k}^{\prime }\left( y\right) +k\left( p-k-1\right)
M_{k}\left( y\right) =0,
\end{equation*}%
are defined by \cite{Masjed}%
\begin{equation*}
M_{k}^{\left( p,q\right) }\left( y\right) =\left( -1\right) ^{k}\Gamma
\left( k+q+1\right) \sum_{m=0}^{k}\frac{\left( -k\right) _{m}\left(
k+1-p\right) _{m}y^{m}}{m!\Gamma \left( m+q+1\right) },
\end{equation*}%
and satisfy the following orthogonality relation%
\begin{equation}
\int\limits_{0}^{\infty }y^{q}\left( 1+y\right) ^{-\left( p+q\right)
}M_{k}^{\left( p,q\right) }\left( y\right) M_{r}^{\left( p,q\right) }\left(
y\right) dy=\frac{k!\Gamma \left( p-k\right) \Gamma \left( q+k+1\right) }{%
\left( p-2k-1\right) \Gamma \left( p+q-k\right) }\delta _{k,r},  \label{Mort}
\end{equation}%
if and only if $q>-1$ and $p>2\left\{ \max k\right\} +1$.

\section{The finite bivariate M - Konhauser polynomials}

We define the finite bivariate M - Konhauser polynomials$\ _{K}M_{k;\upsilon
}^{\left( p,q\right) }\left( y,z\right) $, constructed using the Konhauser
polynomials and the finite univariate orthogonal polynomials $M_{k}^{\left(
p,q\right) }\left( y\right) $.

\begin{definition}
The finite bivariate M - Konhauser polynomials are defined as%
\begin{equation}
\ _{K}M_{k;\upsilon }^{\left( p,q\right) }\left( y,z\right) =\left(
-1\right) ^{k}\Gamma \left( k+q+1\right) \sum_{l=0}^{k}\sum_{s=0}^{k-l}\frac{%
\left( -k\right) _{l+s}\left( k+1-p\right) _{l}\left( -y\right)
^{l}z^{\upsilon s}}{\Gamma \left( l+q+1\right) \Gamma \left( \upsilon
s-p-q+1\right) l!s!}  \label{MKdef}
\end{equation}%
where $p>2\max \left\{ k\right\} +1$, $q>-1$ and $\upsilon =1,2,...$.
\end{definition}

\begin{remark}
For $\upsilon =1$, by choosing $z=0$ the finite bivariate M - Konhauser
polynomials gives the finite univariate orthogonal polynomials $%
M_{k}^{\left( p,q\right) }\left( y\right) $, that is%
\begin{equation*}
\ _{K}M_{k;1}^{\left( p,q\right) }\left( y,0\right) =\frac{1}{\Gamma \left(
1-p-q\right) }M_{k}^{\left( p,q\right) }\left( y\right) .
\end{equation*}
\end{remark}

\begin{theorem}
The finite bivariate M - Konhauser polynomials$\ _{K}M_{k;\upsilon }^{\left(
p,q\right) }\left( y,z\right) $ can be expressed in terms of polynomials $%
Z_{k}^{\left( c\right) }\left( z;\upsilon \right) $ of form%
\begin{equation}
\ _{K}M_{k;\upsilon }^{\left( p,q\right) }\left( y,z\right) =\left(
-1\right) ^{k}k!\Gamma \left( k+q+1\right) \sum_{l=0}^{k}\frac{\left(
-1\right) ^{l}\left( k+1-p\right) _{l}\left( -y\right) ^{l}Z_{k-l}^{\left(
-p-q\right) }\left( z;\upsilon \right) }{l!\Gamma \left( l+q+1\right) \Gamma
\left( \upsilon \left( k-l\right) -p-q+1\right) }.  \label{MZrel}
\end{equation}
\end{theorem}

\begin{proof}
The proof is completed from the definitions of the polynomials $%
_{K}M_{k;\upsilon }^{\left( p,q\right) }\left( y,z\right) $ and $%
Z_{k}^{\left( c\right) }\left( z;\upsilon \right) $.
\end{proof}

On the other hand, $F_{k}\left( y,z\right) $ is as follow%
\begin{equation}
F_{k}\left( y,z\right) :=\ _{K}\mathcal{M}_{k;\upsilon }^{\left( p,q\right)
}\left( y,z\right) =M_{k}^{\left( p,q\right) }\left( y\right)
\sum\limits_{j=0}^{k}Y_{j}^{\left( -p-q\right) }\left( z;\upsilon \right) .
\label{Qdef}
\end{equation}

\begin{corollary}
By Theorem 3, we deduce that the finite bivariate M - Konhauser polynomials $%
_{K}M_{k;\upsilon }^{\left( p,q\right) }\left( y,z\right) $ and the
polynomials$\ _{K}\mathcal{M}_{k;\upsilon }^{\left( p,q\right) }\left(
y,z\right) $ together are a pair of finite bivariate biorthogonal
polynomials with $\rho \left( y,z\right) =y^{q}\left( 1+y\right) ^{-\left(
p+q\right) }e^{-z}z^{-p-q}$ on $\left( 0,\infty \right) \times \left(
0,\infty \right) $.
\end{corollary}

\begin{theorem}
The finite bivariate M - Konhauser polynomials $_{K}M_{k;\upsilon }^{\left(
p,q\right) }\left( y,z\right) $ have the finite biorthogonality relation%
\begin{equation}
\int\limits_{0}^{\infty }\int\limits_{0}^{\infty }y^{q}\left( 1+y\right)
^{-\left( p+q\right) }e^{-z}z^{-p-q}\ _{K}M_{k;\upsilon }^{\left( p,q\right)
}\left( y,z\right) \ _{K}\mathcal{M}_{r;\upsilon }^{\left( p,q\right)
}\left( y,z\right) dzdy=\frac{k!\Gamma \left( p-k\right) \Gamma \left(
q+k+1\right) }{\left( p-2k-1\right) \Gamma \left( p+q-k\right) }\delta
_{k,r},  \label{MKort}
\end{equation}%
where $k,r=0,1,...,R<\frac{p-1}{2},\ q>-1$.
\end{theorem}

\begin{proof}
With the help of definitions (\ref{MZrel}) and (\ref{Qdef}), we write%
\begin{align*}
& \int\limits_{0}^{\infty }\int\limits_{0}^{\infty }y^{q}\left( 1+y\right)
^{-\left( p+q\right) }e^{-z}z^{-p-q}\ _{K}M_{k;\upsilon }^{\left( p,q\right)
}\left( y,z\right) \ _{K}\mathcal{M}_{r;\upsilon }^{\left( p,q\right)
}\left( y,z\right) dydz \\
& =\left( -1\right) ^{k}\Gamma \left( k+q+1\right) \sum_{s=0}^{k}\frac{%
\left( -1\right) ^{s}\left( -k\right) _{s}\left( k-p+1\right) _{s}}{s!\Gamma
\left( s+q+1\right) }\int\limits_{0}^{\infty }y^{q+s}\left( 1+y\right)
^{-\left( p+q\right) }M_{r}^{\left( p,q\right) }\left( y\right) dy \\
& \times \frac{\left( k-s\right) !}{\Gamma \left( \upsilon \left( k-s\right)
-p-q+1\right) }\sum_{j=0}^{r}\int\limits_{0}^{\infty
}e^{-z}z^{-p-q}Z_{k-s}^{\left( -p-q\right) }\left( z;\upsilon \right)
Y_{j}^{\left( -p-q\right) }\left( z;\upsilon \right) dz \\
& =\int\limits_{0}^{\infty }y^{q}\left( 1+y\right) ^{-\left( p+q\right)
}M_{r}^{\left( p,q\right) }\left( y\right) M_{k}^{\left( p,q\right) }\left(
y\right) dy.
\end{align*}%
Using (\ref{Mort}), the proof is completed.
\end{proof}

The Kampe de Feriet's double hypergeometric functions \cite{SD} are defined
as%
\begin{equation*}
F_{k,s,r}^{n,p,q}\left( \QATOP{\left( a_{n}\right) :\left( b_{p}\right)
;\left( c_{q}\right) ;}{\left( d_{k}\right) :\left( f_{s}\right) ;\left(
g_{r}\right) ;}y;z\right) =\sum\limits_{m,l=0}^{\infty }\frac{%
\prod\limits_{i=0}^{n}\left( a_{i}\right) _{m+l}\prod\limits_{i=0}^{p}\left(
b_{i}\right) _{m}\prod\limits_{i=0}^{q}\left( c_{i}\right) _{l}\ y^{l}z^{m}}{%
\prod\limits_{i=0}^{k}\left( d_{i}\right) _{m+l}\prod\limits_{i=0}^{s}\left(
f_{i}\right) _{m}\prod\limits_{i=0}^{r}\left( f_{i}\right) _{l}\ l!m!},
\end{equation*}%
and thus we get the following series representation for the polynomials$\
_{K}M_{k;\upsilon }^{\left( p,q\right) }\left( y,z\right) $.

\begin{theorem}
The finite bivariate M - Konhauser polynomials$\ _{K}M_{k;\upsilon }^{\left(
p,q\right) }\left( y,z\right) $ can be written in terms of the Kampe de
Feriet's double hypergeometric functions by%
\begin{equation*}
\ _{K}M_{k;\upsilon }^{\left( p,q\right) }\left( y,z\right) =\frac{\left(
-1\right) ^{k}\left( q+1\right) _{k}}{\Gamma \left( 1-p-q\right) }%
F_{0,1,\upsilon }^{1,1,0}\left( \QATOP{-k:k-p+1;-;}{-:q+1;\Delta \left(
\upsilon ;1-p-q\right) ;}-y;\left( \frac{z}{\upsilon }\right) ^{\upsilon
}\right) ,
\end{equation*}%
where $\Delta \left( \upsilon ;\sigma \right) $ denotes the $\upsilon $
parameters $\frac{\sigma }{\upsilon },\frac{\sigma +1}{\upsilon },...,\frac{%
\sigma +\upsilon -1}{\upsilon }$.
\end{theorem}

\begin{proof}
We use the identity $\left( 1+c\right) _{\upsilon r}=\upsilon ^{\upsilon
r}\prod\limits_{j=0}^{\upsilon -1}\left( \frac{c+1+j}{\upsilon }\right) _{r}$
in the definition (\ref{MKdef}) for$\ _{K}M_{k;\upsilon }^{\left( p,q\right)
}\left( y,z\right) $.
\end{proof}

\bigskip

Now, let us recall the definition of the bivariate Jacobi Konhauser
Mittag-Leffler functions $E_{p,q,\upsilon }^{\left( \gamma _{1};\gamma
_{2}\right) }\left( y,z\right) $ in below.

The bivariate Jacobi Konhauser Mittag-Leffler functions $E_{p,q,\upsilon
}^{\left( \gamma _{1};\gamma _{2}\right) }\left( y,z\right) $ is defined by
the formula \cite{OzEl2}%
\begin{equation}
E_{p,q,\upsilon }^{\left( \gamma _{1};\gamma _{2}\right) }\left( y,z\right)
=\sum_{l=0}^{\infty }\sum_{m=0}^{\infty }\frac{\left( \gamma _{1}\right)
_{m+l}\left( \gamma _{2}\right) _{l}y^{l}z^{\upsilon m}}{\Gamma \left(
p+l\right) \Gamma \left( q+\upsilon m\right) l!m!},  \label{Edef}
\end{equation}%
where $\upsilon ,p,q,\gamma _{1},\gamma _{2}\in 
%TCIMACRO{\U{2102} }%
%BeginExpansion
\mathbb{C}
%EndExpansion
$, $\func{Re}\left( \upsilon \right) >0,\func{Re}\left( p\right) ,\func{Re}%
\left( q\right) ,\func{Re}\left( \gamma _{1}\right) >0,\func{Re}\left(
\gamma _{2}\right) >0.$

\bigskip

Thus, the following relation can be given between $_{K}M_{k;\upsilon
}^{\left( p,q\right) }\left( y,z\right) $ and $E_{p,q,\upsilon }^{\left(
\gamma _{1};\gamma _{2}\right) }\left( y,z\right) $.

\begin{theorem}
The following relation is satisfied between (\ref{MKdef}) and (\ref{Edef}):%
\begin{equation}
_{K}M_{k;\upsilon }^{\left( p,q\right) }\left( y,z\right) =\left( -1\right)
^{k}\Gamma \left( k+q+1\right) E_{q+1,1-p-q,\upsilon }^{\left(
-k;k+1-p\right) }\left( -y,z\right) .  \label{MErel}
\end{equation}
\end{theorem}

\begin{proof}
Considering the definitions (\ref{MKdef}) and (\ref{Edef}) together, (\ref%
{MErel}) can be seen easily.
\end{proof}

\bigskip

On the other hand, we know the following relation \cite{Masjed} between the
classical orthogonal Jacobi polynomials and the finite univariate orthogonal
polynomials $M_{k}^{\left( p,q\right) }\left( y\right) $:%
\begin{align*}
M_{k}^{\left( p,q\right) }\left( y\right) =\left( -1\right)
^{k}k!P_{k}^{\left( q,-p-q\right) }\left( 2y+1\right) & \\
\Leftrightarrow P_{k}^{\left( p,q\right) }\left( y\right) =\frac{\left(
-1\right) ^{k}}{k!}M_{k}^{\left( -p-q,p\right) }\left( \frac{y-1}{2}\right)
.&
\end{align*}%
In the present paper, we present a similar relation between the finite
bivariate M - Konhauser polynomials and the bivariate Jacobi Konhauser
polynomials in the following theorem.

\begin{theorem}
The following relation holds between the pairs for the finite bivariate
biorthogonal M - Konhauser polynomials and the bivariate biorthogonal Jacobi
Konhauser polynomials:%
\begin{align}
\ _{K}M_{k;\upsilon }^{\left( p,q\right) }\left( y,z\right) =\left(
-1\right) ^{k}k!\ _{K}P_{k;\upsilon }^{\left( q,-p-q\right) }\left(
2y+1,z\right) &  \label{MJrel} \\
\Leftrightarrow \ _{K}P_{k;\upsilon }^{\left( p,q\right) }\left( y,z\right) =%
\frac{\left( -1\right) ^{k}}{k!}\ _{K}M_{k;\upsilon }^{\left( -p-q,p\right)
}\left( \frac{y-1}{2},z\right) &  \notag
\end{align}%
and%
\begin{align*}
\ _{K}\mathcal{M}_{k;\upsilon }^{\left( p,q\right) }\left( y,z\right)
=\left( -1\right) ^{k}k!\ Q_{k;\upsilon }^{\left( q,-p-q\right) }\left(
2y+1,z\right) & \\
\Leftrightarrow \ Q_{k;\upsilon }^{\left( p,q\right) }\left( y,z\right) =%
\frac{\left( -1\right) ^{k}}{k!}\ _{K}\mathcal{M}_{k;\upsilon }^{\left(
-p-q,p\right) }\left( \frac{y-1}{2},z\right) .&
\end{align*}
\end{theorem}

\subsection{Operational and integral representation and some other
properties for the finite bivariate biorthogonal M - Konhauser polynomials}

It is known that%
\begin{equation*}
D_{y}^{-1}f\left( y\right) =\int\limits_{0}^{y}f\left( \xi \right) d\xi .
\end{equation*}%
\bigskip Operational representation for the bivariate Jacobi Konhauser
polynomials$\ _{K}P_{k}^{\left( p,q\right) }\left( y,z\right) $ is given by 
\cite{OzEl2}%
\begin{align}
\ _{K}P_{k}^{\left( p,q\right) }\left( y,z\right) & =\frac{\Gamma \left(
p+k+1\right) }{k!}\left( \frac{1-y}{2}\right) ^{-p}z^{-q}\left(
1-D_{z}^{-\upsilon }\right) ^{k}  \label{Poprep} \\
& \times \ _{2}F_{0}\left( -k;k+p+q+1;-;\frac{-2}{D_{y}\left(
1-D_{z}^{-\upsilon }\right) }\right) \left\{ \frac{\left( \frac{1-y}{2}%
\right) ^{p}z^{q}}{\Gamma \left( p+1\right) \Gamma \left( q+1\right) }%
\right\}  \notag \\
& =\frac{\Gamma \left( p+k+1\right) }{k!}z^{-q}\left( \frac{1-y}{2}\right)
^{-p}\left( 1-D_{z}^{-\upsilon }\right) ^{k}  \notag \\
& \times Y_{k}\left( \frac{2}{D_{y}\left( 1-D_{z}^{-\upsilon }\right) }%
;p+q+2,1\right) \left\{ \frac{\left( \frac{1-y}{2}\right) ^{p}z^{q}}{\Gamma
\left( p+1\right) \Gamma \left( q+1\right) }\right\} ,  \notag
\end{align}%
where $_{2}F_{0}$ is a special case of the generalized hypergeometric
function and $Y_{k}\left( z;a,b\right) =\ _{2}F_{0}\left[ -k,k+a-1;-,\frac{-z%
}{b}\right] $ are the Bessel polynomials.

\begin{theorem}
The finite bivariate M - Konhauser polynomials $_{K}M_{k;\upsilon }^{\left(
p,q\right) }\left( y,z\right) $ have the following operational representation%
\begin{align*}
\ _{K}M_{k;\upsilon }^{\left( p,q\right) }\left( y,z\right) & =\left(
-1\right) ^{k}\Gamma \left( q+1+k\right) y^{-q}z^{p+q}\left(
1-D_{z}^{-\upsilon }\right) ^{k} \\
& \times \ _{2}F_{0}\left[ -k,k+1-p;-;\frac{-1}{D_{y}\left(
1-D_{z}^{-\upsilon }\right) }\right] \left\{ \frac{y^{q}z^{-p-q}}{\Gamma
\left( q+1\right) \Gamma \left( 1-p-q\right) }\right\} .
\end{align*}
\end{theorem}

\begin{proof}
The proof follows by (\ref{MJrel}) and (\ref{Poprep}).
\end{proof}

For the finite bivariate M - Konhauser polynomials $_{K}M_{k;\upsilon
}^{\left( p,q\right) }\left( y,z\right) $, the following operational
representation in terms of Bessel polynomials holds:%
\begin{align*}
\ _{K}M_{k;\upsilon }^{\left( p,q\right) }\left( y,z\right) & =\left(
-1\right) ^{k}\Gamma \left( k+q+1\right) y^{-q}z^{p+q}\left(
1-D_{z}^{-\upsilon }\right) ^{k} \\
& \times Y_{k}\left( \frac{1}{D_{y}\left( 1-D_{z}^{-\upsilon }\right) }%
;2-p,1\right) \left\{ \frac{y^{q}z^{-p-q}}{\Gamma \left( q+1\right) \Gamma
\left( 1-p-q\right) }\right\} .
\end{align*}%
\bigskip Let us recall the definitions for the incomplete Gamma and the
Gamma functions \cite{Erdelyi, WW}:%
\begin{equation}
\frac{1}{\Gamma \left( y\right) }=\frac{1}{2\pi i}\int\limits_{-\infty
}^{0^{+}}e^{u}u^{-y}du,\ \ \ \ \ \left( \left\vert \arg \left( u\right)
\right\vert \leq \pi \right)  \label{PGamma}
\end{equation}%
and%
\begin{equation}
\Gamma \left( y\right) =\int\limits_{0}^{\infty }e^{-u}u^{y-1}du,\ \ \ \ \
\left( \func{Re}\left( y\right) >0\right) .  \label{Gamma}
\end{equation}%
\bigskip Integral representation holds for$\ _{K}P_{k}^{\left( p,q\right)
}\left( y,z\right) $ is given by \cite{OzEl2}%
\begin{align}
\ _{K}P_{k}^{\left( p,q\right) }\left( y,z\right) & =\frac{-1}{4\pi ^{2}k!}%
\frac{\Gamma \left( p+1+k\right) }{\Gamma \left( p+q+1+k\right) }%
\int\limits_{-\infty }^{0^{+}}\int\limits_{-\infty
}^{0^{+}}\int\limits_{0}^{\infty
}e^{-u_{1}+w_{1}+w_{2}}u_{1}^{k+p+q}w_{1}^{-p-1}w_{2}^{-q-1}  \label{Pintrep}
\\
& \times \left( 1-\left( \frac{z}{w_{2}}\right) ^{\upsilon }-\frac{%
u_{1}\left( 1-y\right) }{2w_{1}}\right) ^{k}du_{1}dw_{1}dw_{2}.  \notag
\end{align}

\begin{theorem}
The finite bivariate M - Konhauser polynomials have the integral
representation%
\begin{align*}
\ _{K}M_{k;\upsilon }^{\left( p,q\right) }\left( y,z\right) & =\frac{\left(
-1\right) ^{k+1}\Gamma \left( k+q+1\right) }{4\pi ^{2}\Gamma \left(
k+1-p\right) }\int\limits_{-\infty }^{0^{+}}\int\limits_{-\infty
}^{0^{+}}\int\limits_{0}^{\infty
}e^{-u_{1}+w_{1}+w_{2}}u_{1}^{k-p}w_{1}^{-q-1}w_{2}^{p+q-1} \\
& \times \left( \frac{w_{2}^{\upsilon }-z^{\upsilon }}{w_{2}^{\upsilon }}+%
\frac{u_{1}y}{w_{1}}\right) ^{k}du_{1}dw_{1}dw_{2}.
\end{align*}
\end{theorem}

\begin{proof}
The proof can be easily seen with the help of (\ref{Pintrep}) and (\ref%
{MJrel}).
\end{proof}

\bigskip

For the multiplication of the bivariate Jacobi Konhauser polynomials, the
integral representation \cite{OzEl2}%
\begin{align}
\ _{K}P_{r}^{\left( p,q\right) }\left( y,z\right) \ _{K}P_{k}^{\left(
\lambda ,\mu \right) }\left( y,z\right) =\frac{\Gamma \left( r+p+1\right)
\Gamma \left( k+\lambda +1\right) }{16\pi ^{4}r!k!\Gamma \left(
r+p+q+1\right) \Gamma \left( k+\lambda +\mu +1\right) }\ \ \ \ \ & 
\label{A} \\
\times \int\limits_{0}^{\infty }\int\limits_{0}^{\infty
}\int\limits_{-\infty }^{0^{+}}\int\limits_{-\infty
}^{0^{+}}\int\limits_{-\infty }^{0^{+}}\int\limits_{-\infty
}^{0^{+}}e^{-u_{1}-u_{2}+w_{1}+w_{2}+w_{3}+w_{4}}u_{1}^{r+p+q}u_{2}^{k+%
\lambda +\mu }w_{1}^{-\left( p+1\right) }w_{2}^{-\left( q+1\right)
}w_{3}^{-\left( \lambda +1\right) }w_{4}^{-\left( \mu +1\right) }&   \notag
\\
\times \left( 1-\frac{z^{\upsilon }}{w_{2}^{\upsilon }}-\frac{u_{1}\left(
1-y\right) }{2w_{1}}\right) ^{r}\left( 1-\frac{z^{\upsilon }}{%
w_{4}^{\upsilon }}-\frac{u_{2}\left( 1-y\right) }{2w_{3}}\right)
^{k}dw_{4}dw_{3}dw_{2}dw_{1}du_{2}du_{1}&   \notag
\end{align}%
holds true.

\begin{theorem}
The multiplication of the finite bivariate M - Konhauser polynomials have
the integral representation%
\begin{align*}
\ _{K}M_{r;\upsilon }^{\left( p,q\right) }\left( y,z\right) \
_{K}M_{k;\upsilon }^{\left( \lambda ,\mu \right) }\left( y,z\right) =\frac{%
\left( -1\right) ^{r+k}\Gamma \left( r+q+1\right) \Gamma \left( k+\mu
+1\right) }{16\pi ^{4}\Gamma \left( r+1-p\right) \Gamma \left( k+1-\lambda
\right) }\ \ \ \ \ \ \ \ \ \ \ \ \ \ \ \ \ \ \ \ \ \ \ \ \ \ \ \ \ \ & \\
\times \int\limits_{0}^{\infty }\int\limits_{0}^{\infty
}\int\limits_{-\infty }^{0^{+}}\int\limits_{-\infty
}^{0^{+}}\int\limits_{-\infty }^{0^{+}}\int\limits_{-\infty
}^{0^{+}}e^{-u_{1}-u_{2}+w_{1}+w_{2}+w_{3}+w_{4}}u_{1}^{r-p}u_{2}^{k-\lambda
}w_{1}^{-\left( q+1\right) }w_{2}^{-\left( 1-p-q\right) }w_{3}^{-\left( \mu
+1\right) }w_{4}^{-\left( 1-\lambda -\mu \right) }& \\
\times \left( 1-\frac{z^{\upsilon }}{w_{2}^{\upsilon }}-\frac{u_{1}\left(
1-y\right) }{2w_{1}}\right) ^{r}\left( 1-\frac{z^{\upsilon }}{%
w_{4}^{\upsilon }}-\frac{u_{2}\left( 1-y\right) }{2w_{3}}\right)
^{k}dw_{4}dw_{3}dw_{2}dw_{1}du_{2}du_{1}.&
\end{align*}
\end{theorem}

\begin{proof}
From (\ref{A}) and (\ref{MJrel}), the proof is completed.\bigskip
\end{proof}

Bivariate Jacobi Konhauser polynomials have the following generating
function \cite{OzEl2}%
\begin{align}
\sum\limits_{k=0}^{\infty }\frac{\Gamma \left( q+1\right) \left(
p+q+1\right) _{k}}{\left( p+1\right) _{k}}\ _{K}P_{k}^{\left( p,q\right)
}\left( y,z\right) w^{k}=\frac{1}{\left( 1-w\right) ^{p+q+1}}\ \ \ \ \ \ \ \
\ \ \ \ \ \ \ \ \ \ \ \ \ \ \ \ &  \label{Jdog} \\
\times S_{0:0;0}^{1:0;0}\left[ \QATOP{\left[ p+q+1:2,1\right] :-;-;}{-:\left[
p+1:1\right] ;\left[ q+1:\upsilon \right] ;}\frac{w\left( y-1\right) }{%
2\left( 1-w\right) ^{2}},\frac{wz^{\upsilon }}{w-1}\right] ,&  \notag
\end{align}%
where $S_{k:l;m}^{K:L;M}$ denotes the double hypergeometric series \cite{SD2}%
.

\begin{theorem}
The finite bivariate M - Konhauser polynomials have the following generating
function%
\begin{align*}
\sum\limits_{k=0}^{\infty }\frac{\left( -1\right) ^{k}\left( 1-p\right) _{k}%
}{\left( q+1\right) _{k}}\ _{K}M_{k;\upsilon }^{\left( p,q\right) }\left(
y,z\right) \frac{w^{k}}{k!}=\frac{1}{\Gamma \left( 1-p-q\right) \left(
1-w\right) ^{1-p}}\ \ \ \ \ \ \ \ \ \ \ \ \ \ \ & \\
\times S_{0:0;0}^{1:0;0}\left[ \QATOP{\left[ 1-p:2,1\right] :-;-;}{-:\left[
q+1:1\right] ;\left[ 1-p-q:\upsilon \right] ;}\frac{wy}{\left( w-1\right)
^{2}},\frac{wz^{\upsilon }}{w-1}\right] .&
\end{align*}
\end{theorem}

\begin{proof}
(\ref{Jdog}) and (\ref{MJrel}) are used for the proof.\bigskip
\end{proof}

\begin{theorem}
The finite bivariate polynomials$\ _{K}\mathcal{M}_{k;\upsilon }^{\left(
p,q\right) }\left( y,z\right) $ satisfy the following partial differential
equation:%
\begin{equation*}
\left( \left( \frac{\partial }{\partial z}-1\right) ^{\upsilon }\left( z%
\frac{\partial }{\partial z}-z-\upsilon \right) -z\left( \frac{\partial }{%
\partial z}-1\right) ^{\upsilon +1}\right) \ _{K}\mathcal{M}_{k;\upsilon
}^{\left( p,q\right) }\left( y,z\right) =0.
\end{equation*}
\end{theorem}

\begin{proof}
In partial differential equation for $Y_{k}^{\left( c\right) }\left(
z;\upsilon \right) $ \cite[eq. (28)]{Konhauser} given by%
\begin{equation*}
\left( D-1\right) ^{\upsilon }z\left( D-1\right) Y_{k}^{\left( c\right)
}\left( z;\upsilon \right) =\left( z\left( D-1\right) ^{\upsilon
+1}+\upsilon \left( D-1\right) ^{\upsilon }\right) Y_{k}^{\left( c\right)
}\left( z;\upsilon \right) ,
\end{equation*}%
after multiplying the partial\ sum of $Y_{k}^{\left( c\right) }\left(
z;\upsilon \right) $ with the polynomials $M_{k}^{\left( p,q\right) }\left(
y\right) $, we use definition (\ref{Qdef}).
\end{proof}

\subsection{Laplace transform and fractional calculus operators for$\
_{K}M_{k;\protect\upsilon }^{\left( p,q\right) }\left( y,wz\right) $}

The Laplace transform for a one-variable function $f$ is%
\begin{equation}
\mathcal{L}[f]\left( y\right) =\int\limits_{0}^{\infty }e^{-y\xi }f\left(
\xi \right) d\xi ,\ \func{Re}\left( y\right) >0.  \label{LaplaceDef}
\end{equation}

\begin{theorem}
For $\left\vert \frac{w^{\upsilon }}{a^{\upsilon }}\right\vert <1$, the
Laplace transform of the finite bivariate M - Konhauser polynomials $%
_{K}M_{k;\upsilon }^{\left( p,q\right) }\left( y,z\right) $ is obtained as
follow%
\begin{equation*}
\mathcal{L}\left\{ z^{-p-q}\ _{K}M_{k;\upsilon }^{\left( p,q\right) }\left(
y,wz\right) \right\} =\left( q+1\right) _{k}a^{p+q-1}\left( \frac{%
w^{\upsilon }-a^{\upsilon }}{a^{\upsilon }}\right) ^{k}\ _{2}F_{1}\left[ 
\QATOP{-k,k+1-p}{q+1};\frac{ya^{\upsilon }}{w^{\upsilon }-a^{\upsilon }}%
\right] .
\end{equation*}
\end{theorem}

\begin{proof}
\begin{align*}
& \mathcal{L}\left\{ z^{-p-q}\ _{K}M_{k;\upsilon }^{\left( p,q\right)
}\left( y,wz\right) \right\} \\
& =\left( -1\right) ^{k}\Gamma \left( k+q+1\right)
\sum_{l=0}^{k}\sum_{s=0}^{k-l}\frac{\left( -k\right) _{l}\left( -k+l\right)
_{s}\left( k+1-p\right) _{l}\left( -y\right) ^{l}w^{\upsilon s}}{l!s!\Gamma
\left( l+q+1\right) \Gamma \left( \upsilon s-p-q+1\right) }%
\int\limits_{0}^{\infty }e^{-az}z^{\upsilon s-p-q}dz \\
& =a^{p+q-1}\left( 1-\frac{w^{\upsilon }}{a^{\upsilon }}\right) ^{k}\frac{%
\left( -1\right) ^{k}\Gamma \left( k+q+1\right) }{\Gamma \left( q+1\right) }%
\sum_{l=0}^{k}\frac{\left( -k\right) _{l}\left( k+1-p\right) _{l}}{l!\left(
q+1\right) _{l}}\left( \frac{ya^{\upsilon }}{w^{\upsilon }-a^{\upsilon }}%
\right) ^{l} \\
& =a^{p+q-1}\left( q+1\right) _{k}\left( \frac{w^{\upsilon }-a^{\upsilon }}{%
a^{\upsilon }}\right) ^{k}\ _{2}F_{1}\left[ -k,k+1-p;q+1;\frac{ya^{\upsilon }%
}{w^{\upsilon }-a^{\upsilon }}\right] .
\end{align*}
\end{proof}

\begin{corollary}
For the finite bivariate M - Konhauser polynomials $_{K}M_{k;\upsilon
}^{\left( p,q\right) }\left( y,z\right) $, the following Laplace
representation in terms of finite $M_{k}^{\left( p,q\right) }\left( y\right) 
$ polynomials holds:%
\begin{equation*}
\mathcal{L}\left\{ z^{-p-q}\ _{K}M_{k;\upsilon }^{\left( p,q\right) }\left(
y,wz\right) \right\} =a^{p+q-1}\left( \frac{a^{\upsilon }-w^{\upsilon }}{%
a^{\upsilon }}\right) ^{k}M_{k}^{\left( p,q\right) }\left( \frac{%
ya^{\upsilon }}{a^{\upsilon }-w^{\upsilon }}\right) ,
\end{equation*}%
where $M_{k}^{\left( p,q\right) }\left( y\right) =\left( -1\right) ^{k}k!%
\binom{q+k}{k}\ _{2}F_{1}\left[ -k,k+1-p;q+1;-y\right] $ is one of the
families of the three finite classical univariate orthogonal polynomials.
\end{corollary}

\bigskip

The 2D Laplace transform is defined by \cite{KG}%
\begin{equation*}
\mathcal{L}_{2}\left[ f\left( y,z\right) \right] \left( \zeta ,\kappa
\right) =\int\limits_{0}^{\infty }\int\limits_{0}^{\infty }e^{-\left( \zeta
y+\kappa z\right) }f\left( y,z\right) dydz,\ \func{Re}\left( \zeta \right) ,%
\func{Re}\left( \kappa \right) >0.
\end{equation*}

\begin{theorem}
The bivariate Laplace transform of the finite bivariate M - Konhauser
polynomials $_{K}M_{k;\upsilon }^{\left( p,q\right) }\left( y,z\right) $ is
given by%
\begin{align}
\mathcal{L}_{2}\left\{ y^{q}z^{-p-q}\ _{K}M_{k;\upsilon }^{\left( p,q\right)
}\left( w_{1}y,w_{2}z\right) \right\} & =\frac{\Gamma \left( k+q+1\right) }{%
a^{q+1}b^{1-p-q}}\left( \frac{w_{2}^{\upsilon }-b^{\upsilon }}{b^{\upsilon }}%
\right) ^{k}  \label{MLaplace} \\
& \times \ _{2}F_{0}\left[ -k,k+1-p;-;\frac{w_{1}b^{\upsilon }}{a\left(
w_{2}^{\upsilon }-b^{\upsilon }\right) }\right] .  \notag
\end{align}
\end{theorem}

\begin{proof}
The proof is similar to the proof of \textit{Theorem 22}.
\end{proof}

\begin{corollary}
The bivariate Laplace transform of the finite bivariate M - Konhauser
polynomials $_{K}M_{k;\upsilon }^{\left( p,q\right) }\left( y,z\right) $, in
terms of Bessel polynomials, is given by%
\begin{equation*}
\mathcal{L}_{2}\left\{ y^{q}z^{-p-q}\ _{K}M_{k;\upsilon }^{\left( p,q\right)
}\left( w_{1}y,w_{2}z\right) \right\} =\frac{\Gamma \left( k+q+1\right) }{%
a^{q+1}b^{1-p-q}}\left( \frac{w_{2}^{\upsilon }-b^{\upsilon }}{b^{\upsilon }}%
\right) ^{k}Y_{k}\left( \frac{w_{1}b^{\upsilon }}{a\left( b^{\upsilon
}-w_{2}^{\upsilon }\right) };2-p,1\right) .
\end{equation*}
\end{corollary}

\begin{proof}
The definition for the Bessel polynomials in (\ref{MLaplace}) are applied.
\end{proof}

\bigskip

The Riemann-Liouville fractional integral \cite{Kilbas} is%
\begin{equation*}
\ _{y}\mathbb{I}_{a^{+}}^{\sigma }\left( f\right) =\frac{1}{\Gamma \left(
\sigma \right) }\int\limits_{a}^{y}\left( y-\xi \right) ^{\sigma -1}f\left(
\xi \right) d\xi ,\ \ \ \ f\in L^{1}\left[ a,b\right] ,
\end{equation*}%
where $\sigma \in 
%TCIMACRO{\U{2102} }%
%BeginExpansion
\mathbb{C}
%EndExpansion
$, $\func{Re}\left( \sigma \right) >0$ and $y>a$.

The Riemann-Liouville fractional derivative is \cite{Kilbas}%
\begin{equation*}
\ _{y}D_{a^{+}}^{\sigma }\left( f\right) =\left( \frac{d}{dy}\right) ^{k}\
_{y}\mathbb{I}_{a^{+}}^{k-\sigma }\left( f\right) ,\ \ \ \ f\in C^{k}\left[
a,b\right] ,
\end{equation*}%
where $\sigma \in 
%TCIMACRO{\U{2102} }%
%BeginExpansion
\mathbb{C}
%EndExpansion
$, $\func{Re}\left( \sigma \right) >0$, $y>a$, $k=\left[ \func{Re}\left(
\sigma \right) \right] +1$ and $[\func{Re}\left( \sigma \right) ]$ is the
integral part of $\func{Re}\left( \sigma \right) $.\bigskip

The bivariate Jacobi Konhauser polynomials have the following double
fractional integral representation \cite{OzEl2}%
\begin{align}
& \left( \ _{z}I_{b^{+}}^{\lambda }\ _{y}I_{a^{+}}^{\mu }\right) \left[
\left( y-a\right) ^{p}\left( z-b\right) ^{q}\ _{K}P_{k;\upsilon }^{\left(
p,q\right) }\left( 1-2w_{1}\left( y-a\right) ,w_{2}\left( z-b\right) \right) %
\right]  \label{Pfracint} \\
& =\frac{\Gamma \left( k+p+1\right) }{\Gamma \left( k+p+\mu +1\right) }%
\left( y-a\right) ^{p+\mu }\left( z-b\right) ^{q+\lambda }\
_{K}P_{k;\upsilon }^{\left( p+\mu ,q+\lambda \right) }\left( 1-2w_{1}\left(
y-a\right) ,w_{2}\left( z-b\right) \right) .  \notag
\end{align}

\begin{theorem}
For the finite bivariate M - Konhauser polynomials, we have%
\begin{align*}
& \left( \ _{z}I_{b^{+}}^{\lambda }\ _{y}I_{a^{+}}^{\mu }\right) \left[
\left( y-a\right) ^{q}\left( z-b\right) ^{-p-q}\ _{K}M_{k;\upsilon }^{\left(
p,q\right) }\left( w_{1}\left( y-a\right) ,w_{2}\left( z-b\right) \right) %
\right] \\
& =\frac{\Gamma \left( k+q+1\right) }{\Gamma \left( k+q+\mu +1\right) }%
\left( y-a\right) ^{q+\mu }\left( z-b\right) ^{-p-q+\lambda }\
_{K}M_{k;\upsilon }^{\left( p-\mu -\lambda ,q+\mu \right) }\left(
w_{1}\left( y-a\right) ,w_{2}\left( z-b\right) \right) .
\end{align*}
\end{theorem}

\begin{proof}
The proof is followed by (\ref{Pfracint}) and (\ref{MJrel}).\bigskip
\end{proof}

The bivariate Jacobi Konhauser polynomials have the following double
fractional derivative representation \cite{OzEl2}%
\begin{align}
& \left( \ _{z}D_{b^{+}}^{\lambda }\ _{y}D_{a^{+}}^{\mu }\right) \left[
\left( y-a\right) ^{p}\left( z-b\right) ^{q}\ _{K}P_{k;\upsilon }^{\left(
p,q\right) }\left( 1-2w_{1}\left( y-a\right) ,w_{2}\left( z-b\right) \right) %
\right]  \label{Eoprep} \\
& =\frac{\Gamma \left( k+p+1\right) }{\Gamma \left( k+p-\mu +1\right) }%
\left( y-a\right) ^{p-\mu }\left( z-b\right) ^{q-\lambda }\
_{K}P_{k;\upsilon }^{\left( p-\mu ,q-\lambda \right) }\left( 1-2w_{1}\left(
y-a\right) ,w_{2}\left( z-b\right) \right) .  \notag
\end{align}

\begin{theorem}
For the finite bivariate M - Konhauser polynomials, we have%
\begin{align*}
& \left( \ _{z}D_{b^{+}}^{\lambda }\ _{y}D_{a^{+}}^{\mu }\right) \left[
\left( y-a\right) ^{q}\left( z-b\right) ^{-p-q}\ _{K}M_{k;\upsilon }^{\left(
p,q\right) }\left( w_{1}\left( y-a\right) ,w_{2}\left( z-b\right) \right) %
\right] \\
& =\frac{\Gamma \left( k+q+1\right) }{\Gamma \left( k+q-\mu +1\right) }%
\left( y-a\right) ^{q-\mu }\left( z-b\right) ^{-p-q-\lambda }\
_{K}M_{k;\upsilon }^{\left( p+\mu +\lambda ,q-\mu \right) }\left(
w_{1}\left( y-a\right) ,w_{2}\left( z-b\right) \right) .
\end{align*}
\end{theorem}

\begin{proof}
The results follow by (\ref{Eoprep}) and (\ref{MJrel}).
\end{proof}

\subsection{Fourier transform for the finite bivariate M - Konhauser
polynomials}

\subsubsection{Fourier transform of the finite bivariate biorthogonal
polynomials suggested by the Konhauser polynomials and the finite univariate
orthogonal polynomials $M_{k}^{\left( p,q\right) }\left( y\right) $}

The Fourier transform for a function $d(y,z)$ in two variables is defined as 
\cite[p. 111, Equ. (7.1)]{Davies}%
\begin{equation}
\tciFourier \left( d\left( y,z\right) \right) =\int\limits_{-\infty
}^{\infty }\int\limits_{-\infty }^{\infty }e^{-i\left( \xi _{1}y+\xi
_{2}z\right) }d\left( y,z\right) dzdy  \label{Fourier}
\end{equation}%
and the corresponding Parseval identity is given by the statement%
\begin{equation}
\int\limits_{-\infty }^{\infty }\int\limits_{-\infty }^{\infty }d\left(
y,z\right) \overline{f\left( y,z\right) }dzdy=\frac{1}{\left( 2\pi \right)
^{2}}\int\limits_{-\infty }^{\infty }\int\limits_{-\infty }^{\infty
}\tciFourier \left( d\left( y,z\right) \right) \overline{\tciFourier \left(
f\left( y,z\right) \right) }d\xi _{1}d\xi _{2},\ d,f\in L^{2}\left( 
%TCIMACRO{\U{211d} }%
%BeginExpansion
\mathbb{R}
%EndExpansion
\right) .  \label{Parseval}
\end{equation}

Now, we define the specific functions%
\begin{equation*}
\QATOPD\{ . {d\left( y,z\right) =\left( 1+e^{y}\right) ^{-\left(
p_{1}+q_{1}\right) }\exp \left( q_{1}y+\left( \frac{1}{2}-p_{1}-q_{1}\right)
z-\frac{e^{z}}{2}\right) \ _{K}M_{k;\upsilon }^{\left( \alpha ,\beta \right)
}\left( e^{y},e^{z}\right) ,}{f\left( y,z\right) =\left( 1+e^{y}\right)
^{-\left( p_{2}+q_{2}\right) }\exp \left( q_{2}y+\left( \frac{1}{2}%
-p_{2}-q_{2}\right) z-\frac{e^{z}}{2}\right) \ _{K}\mathcal{M}_{r;\upsilon
}^{\left( \gamma ,\delta \right) }\left( e^{y},e^{z}\right) ,}
\end{equation*}%
and using the transforms $\exp y=u$, $\frac{\exp z}{2}=x$, the definition (%
\ref{MKdef}), the Gamma integral (\ref{Gamma}) and the Beta integral given by%
\begin{equation}
B(y,z)=\int\limits_{0}^{\infty }x^{y-1}\left( 1-x\right) ^{z-1}dx=\frac{%
\Gamma \left( y\right) \Gamma \left( z\right) }{\Gamma \left( y+z\right) }%
=B\left( z,y\right) ,\ \func{Re}\left( y\right) >0,\func{Re}\left( z\right)
>0,  \label{Betadef}
\end{equation}%
the corresponding Fourier transform is obtained as follows%
\begin{align}
& \tciFourier \left( d\left( y,z\right) \right) =\int\limits_{-\infty
}^{\infty }\int\limits_{-\infty }^{\infty }e^{-i\left( \xi _{1}y+\xi
_{2}z\right) }d\left( y,z\right) dydz  \label{Fourier1} \\
& =\int\limits_{-\infty }^{\infty }\int\limits_{-\infty }^{\infty
}e^{-i\left( \xi _{1}y+\xi _{2}z\right) }e^{q_{1}y+\left( \frac{1}{2}%
-p_{1}-q_{1}\right) z-\frac{e^{z}}{2}}\left( 1+e^{y}\right) ^{-\left(
p_{1}+q_{1}\right) }\ _{K}M_{k;\upsilon }^{\left( \alpha ,\beta \right)
}\left( e^{y},e^{z}\right) dydz  \notag \\
& =2^{\frac{1}{2}-p_{1}-q_{1}-i\xi _{2}}\left( -1\right) ^{k}\Gamma \left(
k+\beta +1\right) \sum_{l=0}^{k}\sum_{s=0}^{k-l}\frac{\left( -1\right)
^{l}\left( -k\right) _{l+s}}{l!s!\Gamma \left( l+\beta +1\right) }  \notag \\
& \times \frac{\left( k+1-\alpha \right) _{l}2^{\upsilon s}}{\Gamma \left(
\upsilon s-\alpha -\beta +1\right) }B\left( l+q_{1}-i\xi _{1},\ p_{1}-l+i\xi
_{1}\right) \int\limits_{0}^{\infty }e^{-x}x^{\upsilon s-p_{1}-q_{1}-i\xi
_{2}-\frac{1}{2}}dx  \notag \\
& =2^{\frac{1}{2}-p_{1}-q_{1}-i\xi _{2}}\left( -1\right) ^{k}\Gamma \left(
k+\beta +1\right) \sum_{l=0}^{k}\sum_{s=0}^{k-l}\frac{\left( -1\right)
^{l}\left( -k\right) _{l+s}\left( k+1-\alpha \right) _{l}2^{\upsilon s}}{%
l!s!\Gamma \left( l+\beta +1\right) \Gamma \left( \upsilon s-\alpha -\beta
+1\right) }  \notag \\
& \times B\left( l+q_{1}-i\xi _{1},\ p_{1}-l+i\xi _{1}\right) \Gamma \left(
\upsilon s-p_{1}-q_{1}-i\xi _{2}+\frac{1}{2}\right) .  \notag
\end{align}%
So, the Fourier transform (\ref{Fourier1}) can be expressed as%
\begin{equation}
\tciFourier \left( d\left( y,z\right) \right) =\left( -1\right) ^{k}\Gamma
\left( k+\beta +1\right) G_{1}\left( p_{1},q_{1};\xi _{1},\xi _{2}\right)
\Psi _{1}\left( k,p_{1},q_{1},\alpha ,\beta ,\upsilon ;\xi _{1},\xi
_{2}\right) ,  \label{F1}
\end{equation}%
where%
\begin{equation*}
G_{1}\left( p_{1},q_{1};\xi _{1},\xi _{2}\right) =2^{\frac{1}{2}%
-p_{1}-q_{1}-i\xi _{2}}B\left( p_{1}+i\xi _{1},q_{1}-i\xi _{1}\right) \Gamma
\left( \frac{1}{2}-p_{1}-q_{1}-i\xi _{2}\right) 
\end{equation*}%
and%
\begin{equation*}
\Psi _{1}\left( k,p_{1},q_{1},\alpha ,\beta ,\upsilon ;\xi _{1},\xi
_{2}\right) =\sum_{l=0}^{k}\sum_{s=0}^{k-l}\frac{\left( -k\right)
_{l+s}\left( k+1-\alpha \right) _{l}\left( q_{1}-i\xi _{1}\right) _{l}\left( 
\frac{1}{2}-p_{1}-q_{1}-i\xi _{2}\right) _{\upsilon s}2^{\upsilon s}}{%
l!s!\left( 1-p_{1}-i\xi _{1}\right) _{l}\Gamma \left( l+\beta +1\right)
\Gamma \left( \upsilon s-\alpha -\beta +1\right) }.
\end{equation*}%
On the other hand, Fourier transform of $f\left( y,z\right) $ is%
\begin{align}
& \tciFourier \left( f\left( y,z\right) \right) =\int\limits_{-\infty
}^{\infty }\int\limits_{-\infty }^{\infty }e^{-i\left( \xi _{1}y+\xi
_{2}z\right) }f\left( y,z\right) dydz  \label{Fourier2} \\
& =\int\limits_{-\infty }^{\infty }\int\limits_{-\infty }^{\infty
}e^{-i\left( \xi _{1}y+\xi _{2}z\right) }e^{q_{2}y+\left( \frac{1}{2}%
-p_{2}-q_{2}\right) z-\frac{e^{z}}{2}}\left( 1+e^{y}\right) ^{-\left(
p_{2}+q_{2}\right) }\ _{K}\mathcal{M}_{r;\upsilon }^{\left( \gamma ,\delta
\right) }\left( e^{y},e^{z}\right) dydz  \notag \\
& =\int\limits_{0}^{\infty }u^{q_{2}-i\xi _{1}-1}\left( 1+u\right)
^{-p_{2}-q_{2}}M_{r}^{\left( \gamma ,\delta \right) }\left( u\right)
du\int\limits_{0}^{\infty }e^{-\frac{t}{2}}t^{-p_{2}-q_{2}-i\xi _{2}-\frac{1%
}{2}}\sum_{j=0}^{r}Y_{j}^{\left( -\gamma ,-\delta \right) }\left( t;\upsilon
\right) dt  \notag \\
& =2^{\frac{1}{2}-p_{2}-q_{2}-i\xi _{2}}\left( -1\right) ^{r}\Gamma \left(
r+\delta +1\right) \sum_{s=0}^{r}\frac{\left( -1\right) ^{s}\left( -r\right)
_{s}\left( r+1-\gamma \right) _{s}}{s!\Gamma \left( s+\delta +1\right) }%
B\left( q_{2}+s-i\xi _{1},\ p_{2}-s+i\xi _{1}\right)   \notag \\
& \times \sum_{j=0}^{r}\sum_{l=0}^{j}\sum_{m=0}^{l}\frac{\left( -1\right)
^{m}2^{l}}{j!l!}\binom{l}{m}\left( \frac{m-\gamma -\delta +1}{\upsilon }%
\right) _{j}\int\limits_{0}^{\infty }e^{-x}x^{l-p_{2}-q_{2}-i\xi _{2}-\frac{1%
}{2}}dx  \notag \\
& =2^{\frac{1}{2}-p_{2}-q_{2}-i\xi _{2}}\left( -1\right) ^{r}\Gamma \left(
r+\delta +1\right) \sum_{s=0}^{r}\frac{\left( -1\right) ^{s}\left( -r\right)
_{s}\left( r+1-\gamma \right) _{s}}{s!\Gamma \left( s+\delta +1\right) }%
B\left( q_{2}+s-i\xi _{1},\ p_{2}-s+i\xi _{1}\right)   \notag \\
& \times \sum_{j=0}^{r}\sum_{l=0}^{j}\sum_{m=0}^{l}\frac{\Gamma \left(
l-p_{2}-q_{2}-i\xi _{2}+\frac{1}{2}\right) \ 2^{l}}{\left( -1\right) ^{m}j!l!%
}\binom{l}{m}\left( \frac{m-\gamma -\delta +1}{\upsilon }\right) _{j}. 
\notag
\end{align}%
Then, $\tciFourier \left( f\left( y,z\right) \right) $ can be written as%
\begin{equation}
\tciFourier \left( f\left( y,z\right) \right) =\left( -1\right) ^{r}\Gamma
\left( r+\delta +1\right) G_{2}\left( p_{2},q_{2};\xi _{1},\xi _{2}\right)
\Psi _{2}\left( r,p_{2},q_{2},\gamma ,\delta ,\upsilon ;\xi _{1},\xi
_{2}\right) ,  \label{F2}
\end{equation}%
where%
\begin{equation*}
G_{2}\left( p_{2},q_{2};\xi _{1},\xi _{2}\right) =2^{\frac{1}{2}%
-p_{2}-q_{2}-i\xi _{2}}B\left( p_{2}+i\xi _{1},q_{2}-i\xi _{1}\right) \Gamma
\left( \frac{1}{2}-p_{2}-q_{2}-i\xi _{2}\right) 
\end{equation*}%
and%
\begin{align*}
\Psi _{2}\left( r,p_{2},q_{2},\gamma ,\delta ,\upsilon ;\xi _{1},\xi
_{2}\right) & =\sum_{s=0}^{r}\frac{\left( -r\right) _{s}\left( r+1-\gamma
\right) _{s}\left( q_{2}-i\xi _{1}\right) _{s}}{s!\left( 1-p_{2}-i\xi
_{1}\right) _{s}\Gamma \left( s+\delta +1\right) } \\
& \times \sum_{j=0}^{r}\sum_{l=0}^{j}\sum_{m=0}^{l}\frac{\left( -1\right)
^{m}2^{l}\left( \frac{1}{2}-p_{2}-q_{2}-i\xi _{2}\right) _{l}}{j!m!\left(
l-m\right) !}\left( \frac{m-\gamma -\delta +1}{\upsilon }\right) _{j}.
\end{align*}

\subsubsection{The set of finite biorthogonal functions obtained from
Fourier transform of the pair of the finite bivariate biorthogonal
polynomials suggested by the Konhauser polynomials and the finite univariate
orthogonal polynomials $M_{k}^{\left( p,q\right) }\left( y\right) $}

By substituting (\ref{F1}) and (\ref{F2}) in Parseval's identity and
substituting $e^{y}=u$ and $e^{z}=t$, we get%
\begin{align}
& \int\limits_{0}^{\infty }\int\limits_{0}^{\infty }u^{q_{1}+q_{2}-1}\left(
1+u\right) ^{-\left( p_{1}+q_{1}+p_{2}+q_{2}\right) }e^{-t}t^{-\left(
p_{1}+q_{1}+p_{2}+q_{2}\right) }\ _{K}M_{k;\upsilon }^{\left( \alpha ,\beta
\right) }\left( u,t\right) \ _{K}\mathcal{M}_{r;\upsilon }^{\left( \gamma
,\delta \right) }\left( u,t\right) dudt  \label{Eq} \\
& =\frac{\left( -1\right) ^{k+r}\Gamma \left( k+\beta +1\right) \Gamma
\left( r+\delta +1\right) }{\left( 2\pi \right) ^{2}}\int\limits_{-\infty
}^{\infty }\int\limits_{-\infty }^{\infty }G_{1}\left( p_{1},q_{1};\xi
_{1},\xi _{2}\right) \overline{G_{2}\left( p_{2},q_{2};\xi _{1},\xi
_{2}\right) }  \notag \\
& \times \Psi _{1}\left( k,p_{1},q_{1},\alpha ,\beta ,\upsilon ;\xi _{1},\xi
_{2}\right) \overline{\Psi _{2}\left( r,p_{2},q_{2},\gamma ,\delta ,\upsilon
;\xi _{1},\xi _{2}\right) }d\xi _{1}d\xi _{2}.  \notag
\end{align}%
Thus, if in the left-hand side of (\ref{Eq}) $p_{1}+p_{2}+1=\alpha =\gamma $
and $q_{1}+q_{2}-1=\beta =\delta $ are taken then according to the
orthogonality relation (\ref{MKort}), the equation (\ref{Eq}) reads as%
\begin{align*}
& \int\limits_{0}^{\infty }\int\limits_{0}^{\infty }u^{q_{1}+q_{2}-1}\left(
1+u\right) ^{-\left( p_{1}+q_{1}+p_{2}+q_{2}\right) }e^{-t}t^{-\left(
p_{1}+q_{1}+p_{2}+q_{2}\right) } \\
& \times \ _{K}M_{k;\upsilon }^{\left( p_{1}+p_{2}+1,q_{1}+q_{2}-1\right)
}\left( u,t\right) \ _{K}\mathcal{M}_{r;\upsilon }^{\left(
p_{2}+p_{1}+1,q_{2}+q_{1}-1\right) }\left( u,t\right) dudt \\
& =\frac{k!\Gamma \left( p_{1}+p_{2}+1-k\right) \Gamma \left(
q_{1}+q_{2}+k\right) }{\left( p_{1}+p_{2}-2k\right) \Gamma \left(
p_{1}+p_{2}+q_{1}+q_{2}-k\right) }\delta _{k,r} \\
& =\frac{\Gamma ^{2}\left( k+q_{1}+q_{2}\right) }{\left( 2\pi \right) ^{2}}%
\int\limits_{-\infty }^{\infty }\int\limits_{-\infty }^{\infty }G_{1}\left(
p_{1},q_{1};\xi _{1},\xi _{2}\right) \overline{G_{2}\left( p_{2},q_{2};\xi
_{1},\xi _{2}\right) } \\
& \times \Psi _{1}\left( k,p_{1},q_{1},p_{1}+p_{2}+1,q_{1}+q_{2}-1,\upsilon
;\xi _{1},\xi _{2}\right) \\
& \times \overline{\Psi _{2}\left(
r,p_{2},q_{2},p_{2}+p_{1}+1,q_{2}+q_{1}-1,\upsilon ;\xi _{1},\xi _{2}\right) 
}d\xi _{1}d\xi _{2}.
\end{align*}%
That is,%
\begin{align*}
& \int\limits_{-\infty }^{\infty }\int\limits_{-\infty }^{\infty
}G_{1}\left( p_{1},q_{1};\xi _{1},\xi _{2}\right) \Psi _{1}\left(
k,p_{1},q_{1},p_{1}+p_{2}+1,q_{1}+q_{2}-1,\upsilon ;\xi _{1},\xi _{2}\right)
\\
& \times \overline{G_{2}\left( p_{2},q_{2};\xi _{1},\xi _{2}\right) }%
\overline{\Psi _{2}\left( r,p_{2},q_{2},p_{2}+p_{1}+1,q_{2}+q_{1}-1,\upsilon
;\xi _{1},\xi _{2}\right) }d\xi _{1}d\xi _{2} \\
& =\frac{\left( 2\pi \right) ^{2}k!\Gamma \left( p_{1}+p_{2}+1-k\right) }{%
\left( p_{1}+p_{2}-2k\right) \Gamma \left( q_{1}+q_{2}+k\right) \Gamma
\left( p_{1}+p_{2}+q_{1}+q_{2}-k\right) }\delta _{k,r}.
\end{align*}

As a result, the following theorem can be given.

\begin{theorem}
For $p_{1},p_{2},q_{1},q_{2}>0,$ $p_{1}+p_{2}>2k$, $p_{1}+q_{1}<1/2$ and $%
p_{2}+q_{2}<1/2$, the families of functions$\ \Upsilon _{1}\left(
k,p_{1},q_{1},q_{2},p_{2},\upsilon ;y,z\right) $ and$\ \Upsilon _{2}\left(
r,p_{2},q_{2},q_{1},p_{1},\upsilon ;y,z\right) $ are finite biorthogonal
functions with respect to the weight function%
\begin{eqnarray*}
\rho \left( p_{1},q_{1},p_{2},q_{2};y,z\right)  &=&\Gamma \left(
p_{1}+iy\right) \Gamma \left( p_{2}-iy\right) \Gamma \left( q_{1}-iy\right)
\Gamma \left( q_{2}+iy\right)  \\
&&\times \Gamma \left( 1/2-p_{1}-q_{1}-iz\right) \Gamma \left(
1/2-p_{2}-q_{2}+iz\right) 
\end{eqnarray*}%
and the corresponding finite biorthogonality relation is%
\begin{align*}
& \int\limits_{-\infty }^{\infty }\int\limits_{-\infty }^{\infty }\Gamma
\left( p_{1}+iy\right) \Gamma \left( p_{2}-iy\right) \Gamma \left(
q_{1}-iy\right) \Gamma \left( q_{2}+iy\right) \Gamma \left( \frac{1}{2}%
-p_{1}-q_{1}-iz\right) \Gamma \left( \frac{1}{2}-p_{2}-q_{2}+iz\right)  \\
& \times \Upsilon _{1}\left( k,p_{1},q_{1},q_{2},p_{2},\upsilon
;iz,it\right) \Upsilon _{2}\left( r,p_{2},q_{2},q_{1},p_{1},\upsilon
;-iz,-it\right) dydz \\
& =\frac{4\pi ^{2}k!\Gamma \left( p_{1}+q_{1}\right) \Gamma \left(
p_{2}+q_{2}\right) \Gamma \left( p_{1}+p_{2}+1-k\right) }{\left(
p_{1}+p_{2}-2k\right) \Gamma \left( k+q_{1}+q_{2}\right) \Gamma \left(
p_{1}+p_{2}+q_{1}+q_{2}-k\right) }\delta _{k,r},
\end{align*}%
where%
\begin{equation*}
\Upsilon _{1}\left( k,p_{1},q_{1},q_{2},p_{2},\upsilon ;y,z\right) =2^{\frac{%
1}{2}-p_{1}-q_{1}}\Psi _{1}\left(
k,p_{1},q_{1},p_{1}+p_{2}+1,q_{1}+q_{2}-1,\upsilon ;-iy,-iz\right) 
\end{equation*}%
and%
\begin{equation*}
\Upsilon _{2}\left( r,p_{2},q_{2},q_{1},p_{1},\upsilon ;y,z\right) =2^{\frac{%
1}{2}-p_{2}-q_{2}}\Psi _{2}\left(
r,p_{2},q_{2},p_{2}+p_{1}+1,q_{2}+q_{1}-1,\upsilon ;-iy,-iz\right) .
\end{equation*}
\end{theorem}

\begin{remark}
It is noted that the weight function of this biorthogonality relation is
positive for $p_{1}=p_{2}$ and $q_{1}=q_{2}$.
\end{remark}

\begin{center}
$\mathtt{FUNDING}$
\end{center}

The author G\"{u}ldo\u{g}an Lekesiz E. in the work has been partially
supported by the Scientific and Technological Research Council of T\"{u}%
rkiye (TUBITAK) (Grant number 2218-122C240).\bigskip

\begin{center}
\texttt{DATA AVAILABILITY}
\end{center}

Data sharing not applicable to this article as no datasets were generated or
analysed during the current study.\bigskip

\begin{center}
DECLARATIONS
\end{center}

$\mathbf{Conflict\ of\ Inte}$\textbf{$r$}$\mathbf{est\ }$The authors
declared that they have no conflict of interest.

\end{document}